\documentclass[a4paper,11pt]{article}

\usepackage{amsmath,amsfonts,amssymb,graphics,epsfig,color}

\usepackage[left=2cm,top=3cm,right=2cm,bottom=3cm,bindingoffset=0.5cm]{geometry}
\usepackage[english]{babel}
\usepackage{verbatim}
\usepackage{amscd,enumerate}
\usepackage{epstopdf}
\usepackage{graphicx}
\usepackage{multicol}
\usepackage{appendix}
\usepackage{mathrsfs}
\usepackage{amsmath}
\usepackage{amsthm}
\usepackage{cite}

\newtheorem{thm}{Theorem}[section]

\newtheorem{lem}[thm]{Lemma}

\newtheorem{obs}{Remark}[section]

\newcommand{\dee}{\mathop{\mathrm{d}\!}}

\DeclareMathOperator\erf{erf}

\newcommand\be{\begin{equation}}
\newcommand\ee{\end{equation}}

\newcommand{\bint}{\displaystyle\int}

\title{A class of moving boundary problems with an exponential source term}
\author{
Julieta Bollati $^{1}$, Ernesto A. Borrego Rodriguez $^{1}$, Adriana C. Briozzo $^{1}$, Colin Rogers $^{2}$\\
\small {{$^1$} Depto. Matem\'atica, FCE, Univ. Austral, Paraguay 1950-CONICET} \\
\small {S2000FZF Rosario, Argentina.}\\
\small {{$^2$} School of Mathematics and Statistics, The University of New South Wales},\\
\small {Sydney NSW 2052, Australia}\\
\small{Email: jbollati@austral.edu.ar, eaborrego93@gmail.com, abriozzo@austral.edu.ar, c.rogers@unsw.edu.au}
}
\date{}

\begin{document}

\maketitle
\abstract{This work investigates a class of moving boundary problems related to a nonlinear evolution equation featuring an exponential source term. We establish a connection to Stefan-type problems, for different boundary conditions at the fixed face, through the application of a reciprocal transformation alongside the Cole-Hopf transformation. For specific cases, we derive explicit similarity solutions in parametric form. This innovative approach enhances our understanding of the underlying dynamics and offers valuable insights into the behavior of these systems.}

\smallskip


\smallskip


\section{Introduction}

Stefan-type moving boundary problems naturally arise in the analysis of the melting of solids and the freezing of liquids ( q.v. \cite{Ru1971, Fr1982, ElOc1982, Cr1984, AlSo1993, Ta2000} and literature cited therein). The classical Stefan problem concerns the canonical linear heat equation and where the heat balance requirement on the moving boundary which separates the phases leads to a nonlinear boundary condition in the temperature.

Moving boundary value problems for the Bateman-Burgers equations \cite{Ba1915, Bu1948} linked to the classical heat equation and an integral version \cite{CaLi1989} of the standard Cole-Hopf transformation have been the subject of extensive investigations \cite{CaLi1994, AbLi2000, AbLi2000-2}. In \cite{Ro2015}, such integral representation was conjugated with a reciprocal transformation to derive parametric exact solution to cases of moving boundary problems which arise, in particular, in the context of the percolation of liquids through a porous medium such as soil. The procedure was shown therein to extend, importantly, to a wide class of moving boundary problems which incorporate heterogeneity in the porous medium.

Reciprocal transformations were originally introduced in connection with lift and drag phenomena in gas dynamic (Bateman \cite{Ba1938}). These constituted multi-parameter relations which leave invariant the conservation laws of the governing system, up to the gas laws. It was subsequently established in \cite{Ba1944} that these relations constitute particular B\"acklund transformations the later have their genesis in the classical geometry of pseudo-spherical surfaces. B\"acklund transformation have proved to have diverse important physical applications both in continuous mechanics and modern soliton theory \cite{RoSh1982, RoSc2002, Ro2022}. Invariance of canonical system in both steady and non-steady relativistic gas dynamic has recently been established under multi-parameter reciprocal transformation \cite{RoRu2020, RoRuSc2020}.

In soliton theory, reciprocal transformations of a type introduced in \cite{KiRo1982} have been applied not only to invariance properties \cite{RoNu1986} but also to link the canonical AKNS and WKI inverse scattering schemes \cite{RoWo1984}. Novel reciprocal links between solitonic systems in 2th-dimensions have been established in \cite{OeRo1993}.

Reciprocal transformations have been applied in \cite{Ro1985, Ro1986} to obtain exact solutions to Stefan-type moving boundary problems for a nonlinear heat equation which models conduction in a range of metals as delimited in \cite{St1951}. Such a reciprocal transformation was subsequently used in \cite{RoGu1988} to determine conditions for the onset of melting in such metals subject to applied boundary flux.

The application of reciprocal transformations to solve a range of Stefan-type moving boundary problems which model the percolation of liquids through porous media such as soils has been detailed in \cite{RoBr1988, Ro2019, Ro2019-2}.

In recent work \cite{BrRoTa2023}, application of a reciprocal transformation was made to solve a class of Stefan-type moving boundary problems for a canonical nonlinear evolution equation which incorporates a constant source term. In \cite{BrRo1993}, a class of boundary value problems for a nonlinear transport equation with addition of an exponential source term was shown to admit exact solution with a pair of stationary boundaries. This result was derived via a Lie-B\"acklund analysis and models liquid transport through a soil subject to a volumetric extraction process. In the present work, a class of moving boundary problems for the nonlinear equation of \cite{BrRo1993} involving as exponential source term is solved via a novel application of a reciprocal transformation.

This paper is organized as follows. In Section 2, we establish a connection between one-dimensional Stefan-type problems and a class of moving boundary problems governed by a nonlinear evolution equation that includes an exponential source term. We first examine a Stefan-type problem with a Dirichlet condition applied at the fixed face  $z=0$, followed by an analysis of cases where either a Robin or Neumann condition is imposed on the fixed face. Additionally, we investigate the convergence of the Stefan problem with a convective boundary condition as the coefficient characterizing heat transfer at the fixed face approaches infinity. In Section 3, we analyze specific cases that yield explicit similarity solutions in parametric form. Finally, we present our conclusions, along with an appendix discussing the canonical solitonic reciprocal reduction for a third-order nonlinear evolution equation with an exponential source term.

\section{Connection between Stefan type problems with a class of moving boundary problems with an exponential source term}

In this section we establish a connection between one dimensional Stefan type problems with a class of moving boundary problems governed by a nonlinear evolution equation which incorporates an exponential source term. First, we will consider a Stefan-type problem with Dirichlet condition at the fixed face $z=0$. Then, in the next subsection we will address the case in which a Robin or Neumann type condition is imposed on the fixed face. We also study the convergence of the Stefan problem with a convective boundary condition when the coefficient that characterizes the heat transfer at the fixed face goes to infinity.

\subsection{Stefan type problem with Dirichlet condition}
We consider the following classical Stefan-type problem  ($P_1$) with a Dirichlet condition on the fixed face $z=0$ given by
\begin{subequations}\label{P_1}
\begin{align}
&w_{zz}=w_t,\quad &0<z<s(t),\quad t>0,\label{P1-1}\\
&w_z(s(t),t)=-L(t)s'(t),&\qquad t>0,\label{P1-2}\\
&w(s(t),t)=w_m(t)\qquad & t>0,\label{P1-3}\\
&w(0,t)=v(t),\qquad &t>0,\label{P1-4}\\
&s(0)=0.\label{P1-5}
\end{align}
\end{subequations}
We assume that $L=L(t)$ and $v=v(t)$ are  continuous and positive functions for all $t>0$ and $w_m=w_m(t)$ is a differentiable function with $v(t)>w_m(t)>0$ for $t>0$ and $v(0)=w_m(0)$.

The pair $(w,s)$ is said to be solution of the problem $(P_1)$ if $w$, $w_z$, $w_t$ and $w_{zz}$ are continuous in $D=\lbrace(z,t): 0< z< s(t), \quad t>0\rbrace$ and the conditions $\eqref{P1-1}-\eqref{P1-5}$ are satisfied. The function $z=s(t)$ describes the moving boundary, which, due to its unknown nature, is referred to as the free boundary.

\begin{lem}\label{signos}
    
  If $w$ satisfies $\eqref{P1-1}$ in a bounded $D_T=\lbrace(z,t): 0< z< s(t), \quad 0<t<T\rbrace$ and $w$ is continuous in $D_T\cup B_T$ where $B_T=\lbrace(z,t):z=0,\quad 0\leq t<T\rbrace\cup \lbrace(z,t):z=s(t),\quad 0\leq t<T\rbrace$ then, we have 
  \begin{equation}
      w(z,t)>0 \quad and \quad w_z(z,t)\leq 0\quad  \textit{on} \quad D_T\cup B_T
  \end{equation}
\end{lem}
\begin{proof}
To see \cite{Fr1959}    
\end{proof}
 
 Now, we consider a Hopf-Cole transformation to obtain an equivalent problem to ($P_1$) governed by the Burgers equation \cite{Ho1950,Co1951}
\medskip

\begin{thm}\label{theo_1-2}
If $(w,s)$ is a solution to the problem \eqref{P1-1}-\eqref{P1-5}, then the pair $(x,s)$ with
\be\label{transf x-w}
x(z,t)=-\tfrac{2}{\sigma}\tfrac{w_z(z,t)}{w(z,t)},\qquad 0<z<s(t),\quad t>0, 
\ee
satisfies the  problem ($P_2$) given by

\begin{subequations}
\begin{align}
&x_t=x_{zz}-\sigma x x_z,\qquad &0<z<s(t),\quad t>0,\label{P2-1}\\
&x(s(t),t)=\tfrac{2}{\sigma} \tfrac{L(t)}{w_m(t)}s'(t),&\qquad t>0,\label{P2-2}\\
&x_z(s(t),t)=-\tfrac{2}{\sigma}\tfrac{w'_m(t)}{w_m(t)}+\tfrac{\sigma}{2} x^2(s(t),t) \left( 1-\tfrac{w_m(t)}{L(t)}\right),\qquad & t>0,\label{P2-3}\\
&\bint_0^{s(t)} x(\xi,t) \, d\xi = \frac{2}{\sigma} \ln\left(\tfrac{v(t)}{w_m(t)} \right),\qquad &t>0,  \label{P2-4}\\
&s(0)=0,\label{P2-5}
\end{align}
\end{subequations}
for a fixed parameter $\sigma>0.$

\end{thm}
\begin{proof}
From the definition  $x$ given by \eqref{transf x-w} we get that
\begin{equation}\label{xzproblema1}
    \begin{array}{rl}
    x_t(z,t)&=-\tfrac{2}{\sigma} \left( \tfrac{w_{zt}(z,t)}{ w(z,t)}-\tfrac{w_{z}(z,t)w_t(z,t)}{w^2(z,t)}\right),\\[0.25cm]
x_z(z,t)&=-\frac{2}{\sigma} \left(\frac{w_{zz}(z,t)}{w(z,t)}-\frac{w_z^2 (z,t)}{w^2(z,t)}\right),\\[0.25cm] 
x_{zz}(z,t)&=-\tfrac{2}{\sigma} \left(\tfrac{w_{zzz}(z,t)}{w(z,t)}-\tfrac{3w_z(z,t) w_{zz}(z,t)}{w^2(z,t)}+\tfrac{2 w_z^3(z,t)}{w^3(z,t)}\right).
\end{array}
\end{equation}
Therefore, from the heat equation \eqref{P1-1}, 
 we deduce that  
$$\begin{array}{ll}
   & x_{zz}(z,t)-\sigma x(z,t)x_z(z,t)=   -\tfrac{2}{\sigma} \left(\tfrac{w_{zzz}(z,t)}{w(z,t)}-\tfrac{3w_z(z,t) w_{zz}(z,t)}{w^2(z,t)}+\tfrac{2 w_z^3(z,t)}{w^3(z,t)}\right)\\[0.25cm]
 &-\tfrac{4}{\sigma^2}\left(\tfrac{w_{zz}(z,t) w_z(z,t)}{w^2(z,t)}-\tfrac{w_z^3(z,t)}{w ^3(z,t)} \right)=-\tfrac{2}{\sigma} \left( \tfrac{w_{zzz}(z,t)}{ w(z,t)}-\tfrac{w_z(z,t)w_{zz}(z,t)}{w^2(z,t)}\right)\\[0.25cm]
& =-\tfrac{2}{\sigma} \left( \tfrac{w_{zt}(z,t)}{ w(z,t)}-\tfrac{w_{z}(z,t)w_t(z,t)}{w^2(z,t)}\right)=x_t(z,t).
\end{array}$$
That is to say, the function $x$ satisfies the Burgers equation \eqref{P2-1}.

Conditions \eqref{P1-2} and  \eqref{P1-3} imply
that
\begin{equation}
    x(s(t),t)=-\frac{2}{\sigma} \frac{w_z(s(t),t)}{w(s(t),t)}=\frac{2}{\sigma}\frac{L(t)s'(t)}{w_m(t)},
\end{equation}
hence, condition \eqref{P2-2} holds.

If we derive condition \eqref{P1-3} with respect to time $t$ it follows that
$$w_z(s(t),t) s'(t)+w_t(s(t),t)=w_m'(t),$$
then dividing by $w(s(t),t)$ and taking into account  equation \eqref{P1-1} and condition  \eqref{P1-3} we get that
$$\tfrac{w_{zz}(s(t),t)}{w(s(t),t)}=\tfrac{w_m'(t)}{w_m(t)}-\tfrac{w_z(s(t),t)}{w(s(t),t)} s'(t).$$
Then,
$$\begin{array}{rl}
x_z(s(t),t)&=-\tfrac{2}{\sigma}\left( \tfrac{w_{zz}(s(t),t)}{w(s(t),t)}-\tfrac{w^2_z(s(t),t)}{w^2(s(t),t)}\right)\\[0.25cm]
&=-\tfrac{2}{\sigma}\left(\tfrac{w_m'(t)}{w_m(t)}-\tfrac{w_z(s(t),t)}{w(s(t),t)} s'(t)-\tfrac{w^2_z(s(t),t)}{w^2(s(t),t)}\right) \\[0.25cm]
&=-\tfrac{2}{\sigma}\left(\tfrac{w_m'(t)}{w_m(t)}+\tfrac{\sigma}{2} x(s(t),t) s'(t)-\tfrac{\sigma^2}{4}  x^2(s(t),t)\right) \\[0.25cm]
&=-\tfrac{2}{\sigma}\tfrac{w_m'(t)}{w_m(t)}- x^2(s(t),t) \tfrac{\sigma w_m(t)}{2L(t)}+\tfrac{\sigma}{2}  x^2(s(t),t),
\end{array}$$
i.e., the function $x$ satisfies condition \eqref{P2-3}.

Moreover, from  \eqref{transf x-w} it results that
$$\bint_0^{s(t)} x(\xi,t)\,  d \xi=-\tfrac{2}{\sigma}\left(\ln(w(s(t),t))-\ln (w(0,t))\right),$$
from where we get that
$$w(0,t)=w_m(t) \exp\Big( \tfrac{\sigma}{2} \bint_0^{s(t)} x(\xi,t)\,d \xi\Big).$$
Consequently, from \eqref{P1-4}, we obtain that
$$w_m(t) \exp\Big( \tfrac{\sigma}{2} \bint_0^{s(t)} x(\xi,t)d \xi\Big)=v(t),  $$
and then condition \eqref{P2-4} holds.
\end{proof}

In a similar manner, the reciprocal of the previous theorem can be formulated as follows.

\begin{thm}
If $(x,s)$ is a solution to the problem ($P_2$) given by  \eqref{P2-1}-\eqref{P2-5}, then the pair $(w,s)$ with
\be\label{transf w-x}
w(z,t)=w_m(t) \exp\Big(\tfrac{\sigma}{2}\bint_z^{s(t)} x(\xi,t) \, d \xi \Big), \qquad 0<z<s(t),\quad t>0, \quad \sigma>0,
\ee
satisfies the  problem ($P_1$) given by \eqref{P1-1}-\eqref{P1-5}.

\end{thm}

\begin{proof}
From the transformation  \eqref{transf w-x} we get that
\begin{equation}\label{aux21}
    \begin{array}{ll}
         w_z(z,t)=-\frac{\sigma}{2}w(z,t) x(z,t), \\[0.25cm]
         w_{zz}(z,t)=w(z,t)\left(\frac{\sigma^2}{4}x^2(z,t)-\frac{\sigma}{2}x_z(z,t) \right),\\[0.25cm]
         w_t(z,t)= w(z,t)\left[\frac{w_m'(t)}{w_m(t)} +\frac{\sigma}{2} \bint_z^{s(t)}x_t(\xi,t) d\xi+\frac{\sigma}{2} x(s(t),t)s'(t) \right].
    \end{array}
\end{equation}
Considering \eqref{P2-1} it results that
\begin{equation}\label{aux22}
\begin{array}{ll}
    \bint_z^{s(t)}x_t(\xi,t) \,d\xi= \bint_z^{s(t)} \bigg(x_{\xi \xi}(\xi,t)-\sigma x(\xi,t) x_{\xi}(\xi,t)\bigg)  \,d\xi\\[0.35cm] 
    =x_{z}(s(t),t)-x_{z}(z,t)-\frac{\sigma}{2}x^2(s(t),t)+\frac{\sigma}{2}x^2(z,t).
    \end{array}
\end{equation}
Then by \eqref{aux21} and \eqref{aux22} and condition \eqref{P2-3} it follows that
\begin{equation}
    w_t(z,t)=w(z,t)\left(\tfrac{\sigma^2}{4} x^2(z,t)-\tfrac{\sigma}{2}x_z(z,t) \right)=w_{zz}(z,t),
\end{equation}
obtaining as an immediate consequence that equation \eqref{P1-1} holds.

From  \eqref{transf w-x}  we immediately get that $w(s(t),t)=w_m(t)$.
Then, from condition \eqref{P2-2} and \eqref{aux21} we get \eqref{P1-2}.

Finally, condition \eqref{P2-4} straightforwardly leads to condition \eqref{P1-4}.
\end{proof}
\begin{obs}\label{signo x}
From transformation \eqref{transf x-w} and Lemma \ref{signos} we have that $x(z,t)\geq 0$ for all $0\leq z\leq s(t),\quad t>0$. 

\end{obs}
\begin{obs}\label{RT}
Equation \eqref{P2-1} implies that 
\be\label{dx}
dx=x_z dz+(x_{zz}-\sigma xx_z) dt
\ee and reciprocally we have
\be\label{dz}
dz=\frac{1}{x_z} dx-\frac{1}{x_z}(x_{zz}-\sigma xx_z) dt.
\ee 
\end{obs}

Equation \eqref{dx} leads us to introduce a new function \(\Psi = \Psi(x,t)\) and to define the following Bäcklund reciprocal transformation:

\begin{equation}\label{T3}
x_z(z,t) = \frac{1}{\Psi(x,t)}, \qquad x_t(z,t) = -\frac{\Psi_x(x,t)}{\Psi^3(x,t)} - \frac{\sigma x}{\Psi(x,t)}.
\end{equation}

From this transformation, we can establish the following results:

\begin{thm} \label{theo_2-3}
If $(x,s)$ is a solution to the problem ($P_2$), then $(\Psi,X_0,X_1)$ where $\Psi$ is defined by \eqref{T3} and \begin{equation} \label{XceroXuno}X_0(t)=x(0,t),  \qquad X_1(t)=x(s(t),t), \qquad t>0 \end{equation} satisfy the problem $(P_3)$
given by
\begin{subequations}
\begin{align}
&\Psi_t=\left(- \Psi^{-1}\right)_{xx}+\sigma, & X_0(t)<x<X_1(t), \quad t>0,\label{P3-1}\\
&X_1(t)=\bint_0^t \left( \tfrac{\sigma X_1(\tau)}{\Psi(X_1(\tau),\tau)} \left(\tfrac{ w_m(\tau)}{2L(\tau)}-1\right)-\tfrac{\Psi_x(X_1(\tau),\tau)}{\Psi^3(X_1(\tau),\tau)}\right) \, d\tau,&t>0,\label{P3-2}\\
&\Psi(X_1(t),t)=\Big(-\tfrac{2}{\sigma} \tfrac{w_m'(t)}{w_m(t)}+ \tfrac{\sigma X_1^2(t)}{2}\Big(1-\tfrac{w_m(t)}{L(t)}\Big)\Big)^{-1}, & t>0,\label{P3-3}\\
&\bint_0^t \left(\tfrac{\sigma^2}{4}X_0^2(\tau)-\tfrac{\sigma}{2}\tfrac{1}{\Psi(X_0(\tau),\tau)} \right) \dee \tau =\ln \left( \tfrac{2v(t)}{\sigma w_m(0)}\right), &t>0,\label{P3-4}\\
&X_0(t)=\bint_0^t -\tfrac{\Psi_x(X_0(\tau),\tau)}{\Psi^3(X_0(\tau),\tau)}-\tfrac{\sigma X_0(\tau)}{\Psi(X_0(\tau),\tau)} \dee \tau & t>0,\label{P3-5}
\end{align}
\end{subequations}

\end{thm}
\begin{proof}

    According to  \eqref{dx} and \eqref{T3} it follows that 
    $$dx=\tfrac{1}{\Psi(x,t)} dz+\left(\tfrac{-\Psi_x(x,t)}{\Psi^3(x,t)}-\sigma x\tfrac{1}{\Psi(x,t)}\right) dt$$ 
    or equivalently 
    $$dz=\Psi(x,t) dx+(\tfrac{\Psi_x(x,t)}{\Psi^2(x,t)}+\sigma x) dt,$$ then we obtain
    $$\Psi_t=\left(\tfrac{-\Psi_x(x,t)}{\Psi^2(x,t)}-\sigma x\right)_x$$ and \eqref{P3-1} holds.
    
In addition we can write $$x(z,t)=\bint_0^z x_z(\xi,t)\dee\xi+N(t) $$ and using \eqref{P2-1} we have $$N'(t)=x_{zz}(0,t)-\sigma x(0,t)x_z(0,t)$$ then 
$$x(z,t)=\bint_0^z x_z(\xi,t)\dee\xi+\bint_0^t \left(x_{zz}(0,\tau)-\sigma x(0,\tau)x_z(0,\tau)\right) \dee \tau.$$ 

By defining \(X_0(t) = x(0,t)\) and \(X_1(t) = x(s(t),t)\), we find that for each \(t > 0\), the variable \(x\) varies between \(X_0(t)\) and \(X_1(t)\).
 For $X_0(t)$ we have the following expression $$X_0(t)=x(0,t)=\bint_0^t \left(x_{zz}(0,\tau)-\sigma x(0,\tau)x_z(0,\tau)\right) \dee \tau=\bint_0^t \left(\tfrac{-\Psi_x(X_0(\tau),\tau)}{\Psi^3(X_0(\tau),\tau)}-\tfrac{\sigma X_0(\tau)}{\Psi(X_0(\tau),\tau)}\right) \dee \tau.$$
 This is \eqref{P3-5}.
For $X_1(t)$ we have $$X_1(t)=\bint_0^{s(t)} x_z(\xi,\tau)\dee\xi+X_0(t)$$ then by using \eqref{P2-1} and \eqref{P2-2} we obtain
\[ X^{'}_1(t)=x_z(s(t),t)s'(t)+\bint_0^{s(t)} x_{zt}(\xi,\tau)\dee\xi+X^{'}_0(t)
\]
\[=x_z(s(t),t)s'(t)+\bint_0^{s(t)} \left(x_{\xi\xi}(\xi,\tau)-\sigma x(\xi,\tau) x_z(\xi,\tau)\right)_\xi\dee\xi+X^{'}_0(t)\]
\[=x_z(s(t),t)s'(t)+x_{zz}(s(t),t)-\sigma x(s(t),t)x_z(s(t),t)\]

\[=\tfrac{\sigma w_m(t)}{2L(t)}\tfrac{X_1(t)}{\Psi(X_1(t),t)}-\tfrac{\Psi_x(X_1(t),t)}{\Psi^3(X_1(t),t)}-\tfrac{\sigma X_1(t)}{\Psi(X_1(t),t)}\] this is \eqref{P3-2}.

The equation \eqref{P3-3} follows immediately from \eqref{P2-2}, \eqref{P2-3} and \eqref{T3}.

To prove \eqref{P3-5} we define $R(t)=\exp\left( \tfrac{\sigma}{2}\bint_0^{s(t)} x(\xi,t)\dee \xi\right) $
then we have
\[\ln(R(t))=\tfrac{\sigma}{2}\bint_0^{s(t)} x(\xi,t)\dee \xi \]and
\[\tfrac{R'(t)}{R(t)}=\tfrac{\sigma}{2}\left[x(s(t),t)s'(t)+\bint_0^{s(t)} x_t(\xi,t)\dee \xi\right]\]
\[=\tfrac{\sigma}{2}\left[x(s(t),t)s'(t)+\bint_0^{s(t)} \left(x_{zz}(\xi,t)-\sigma x_{z}(\xi,t)x(\xi,t)\right)\dee \xi\right]\]
\[=\tfrac{\sigma}{2}\left[x(s(t),t)s'(t)+x_{z}(s(t),t)-\tfrac{\sigma}{2}x^{2}(s(t),t)-x_{z}(0,t)+\tfrac{\sigma}{2}x^{2}(0,t)\right]\]

\[=\tfrac{\sigma}{2}\left[\tfrac{-2 w_m^{'}(t)}{\sigma w_m(t)}-x_{z}(0,t)+\tfrac{\sigma}{2}x^{2}(0,t)\right]=- \tfrac{w_m^{'}(t)}{ w_m(t)}-\tfrac{\sigma}{2\Psi(X_0(t),t)}+\tfrac{\sigma^{2}}{4}X_0^{2}(t)\]

Then 
\[\ln(R(t))-\ln(R(0))=\ln(w_m(0))-\ln(w_m(t))- \bint_0^t \left(\tfrac{\sigma}{2\Psi(X_0(\tau),\tau)}-\tfrac{\sigma^{2}}{4}X_0^{2}(\tau) \right)\dee\tau\]
and 
\[R(t)=\tfrac{w_m(0)}{w_m(t)}\exp\left(- \bint_0^t \left(\tfrac{\sigma}{2\Psi(X_0(\tau),\tau)}-\tfrac{\sigma^{2}}{4}X_0^{2}(\tau) \right)\dee\tau\right),\] therefore \eqref{P3-4} follows.
\end{proof}

\begin{obs} \label{signosX0-X1}
    Notice that from Remark \ref{signo x} and  \eqref{XceroXuno} we have
   $x=X_0(t)\geq 0$ and $x=X_1(t)\geq 0$ for all $t>0$. Moreover, taking into account \eqref{P3-2} and \eqref{P3-5} it turns out that $X_0(0)=0$ and $X_1(0)=0$.
\end{obs}

\begin{thm}
If $(\Psi(x,t),X_0(t),X_1(t))$ satisfy the problem $(P_3)$ then $(u(z,t),s(t))$ given by
\begin{equation}\label{defu}
u(z,t)=x, 
\end{equation}with
\begin{equation}\label{defz}  z=\bint_{X_0(t)}^{x} \Psi(\xi,t)\dee\xi \qquad  \textit{and}\qquad s(t)=\bint_{X_0(t)}^{X_1(t)} \Psi(\xi,t)\dee \xi\end{equation}
is a solution to the problem ($P_2$).
\end{thm}
\begin{proof} Given $X_0(t)\leq x\leq X_1(t)$ and $\Psi(x,t)$ from \eqref{defu} can be determine $u(z,t)$ with $z$ given by \eqref{defz}. Moreover, from \eqref{defz} we have
    \[ dz=\Psi(x,t)dx+\left[-\Psi(X_0(t),t)X_0^{'}(t)+\bint_{X_0(t)}^{x} \Psi_t(\xi,t)\dee\xi\right]dt. 
  \]
  Using \eqref{P3-1} we obtain
    \[
    dz=\Psi(x,t)dx+\left[-\Psi(X_0(t),t)X_0^{'}(t)+\tfrac{\Psi_x (x,t)}{\Psi^2(x,t)}+\sigma x -\tfrac{\Psi_x (X_0(t),t)}{\Psi^2(X_0(t),t)}-\sigma X_0(t)\right] dt
  \] and taking into account \eqref{P3-5} follows that
     \[dz=\Psi(x,t)dx+\left[\tfrac{\Psi_x (x,t)}{\Psi^2(x,t)}+\sigma x\right] dt\] or equivalently 
     \[dx=\tfrac{1}{\Psi(x,t)}dz-\tfrac{1}{\Psi(x,t)}\left[\tfrac{\Psi_x (x,t)}{\Psi^2(x,t)}+\sigma x\right] dt\] 
     The fact that $u(z,t)=x$ and $z_x=\Psi$ imply that 
     \[u_z(z,t)=\tfrac{1}{\Psi(x,t)} \qquad u_{zz}(z,t)=-\tfrac{\Psi_x(x,t)}{\Psi^{3}(x,t)}\] then
     \[du=dx=u_z(z,t)dz+\left[u_{zz}(z,t)-\sigma u(z,t)u_z(z,t)\right] dt\] and it results $$u_t=u_{zz}-\sigma u u_z$$this is \eqref{P2-1} holds.

     For $x=X_0(t)$ we obtain $z=0$ and for $x=X_1(t)$ we have $z=s(t)$ which is given by \eqref{defz}.
To prove \eqref{P2-2} we consider \eqref{P3-1} and \eqref{defz} . Follows that $$s'(t)=\Psi(X_1(t),t) X'_1(t)-\Psi(X_0(t),t) X'_0(t)+\bint_{X_0(t)}^{X_1(t)} \Psi_t(\xi,t) \dee \xi$$
$$=\Psi(X_1(t),t) X'_1(t)-\Psi(X_0(t),t) X'_0(t)+\bint_{X_0(t)}^{X_1(t)} \left(-\left(\Psi^{-1}\right)_{\xi\xi}(\xi,t)+\sigma\right) \dee \xi$$
 $$=\Psi(X_1(t),t) X'_1(t)-\Psi(X_0(t),t) X'_0(t)-\left(\Psi^{-1}\right)_{x}(X_1(t),t)+\sigma X_1(t)+\left(\Psi^{-1}\right)_{x}(X_0(t),t)-\sigma X_0(t)$$
 then, taking into account \eqref{P3-2}, \eqref{P3-5}, \eqref{defu} and \eqref{defz} we obtain that $$s'(t)=\tfrac{\sigma w_m(t)}{2L(t)}X_1(t)=\tfrac{\sigma w_m(t)}{2L(t)}u(s(t),t)$$
 Therefore
\eqref{P2-2} holds.

From \eqref{defu} and \eqref{defz} we have
\[u_z(s(t),t)=\tfrac{1}{\Psi(X_1(t),t)},\] by using condition \eqref{P2-2} and \eqref{P3-3} it follows that
\[u_z(s(t),t)=-\tfrac{2}{\sigma} \tfrac{w_m'(t)}{w_m(t)}+ \tfrac{\sigma u^{2}(s(t),t)}{2}\Big(1-\tfrac{w_m(t)}{L(t)}\Big)=-\tfrac{2}{\sigma}\tfrac{w'_m(t)}{w_m(t)}-u(s(t),t)s'(t)+\tfrac{\sigma}{2} u^2(s(t),t),\] this is \eqref{P2-3}.

To prove \eqref{P2-4} we consider $V(t)=\exp\left( \tfrac{\sigma}{2}\bint_0^{s(t)} u(\xi,t)\dee \xi\right) $
then we have
\[\ln(V(t))=\tfrac{\sigma}{2}\bint_0^{s(t)} u(\xi,t)\dee \xi \]and
\[\tfrac{V'(t)}{V(t)}=\tfrac{\sigma}{2}\left[u(s(t),t)s'(t)+\bint_0^{s(t)} u_t(\xi,t)\dee \xi\right]\]
\[=\tfrac{\sigma}{2}\left[u(s(t),t)s'(t)+\bint_0^{s(t)} \left(u_{zz}(\xi,t)-\sigma u_{z}(\xi,t)u(\xi,t)\right)\dee \xi\right]\]
\[=\tfrac{\sigma}{2}\left[u(s(t),t)s'(t)+u_{z}(s(t),t)-\tfrac{\sigma}{2}u^{2}(s(t),t)-u_{z}(0,t)+\tfrac{\sigma}{2}u^{2}(0,t)\right]\]

\[=\tfrac{\sigma}{2}\left[\tfrac{-2 w_m^{'}(t)}{\sigma w_m(t)}-u_{z}(0,t)+\tfrac{\sigma}{2}u^{2}(0,t)\right]=- \tfrac{w_m^{'}(t)}{ w_m(t)}-\tfrac{\sigma}{2\Psi(X_0(t),t)}+\tfrac{\sigma^{2}}{4}X_0^{2}(t)\]

Then 
\[\ln(V(t))-\ln(V(0))=\ln(w_m(0))-\ln(w_m(t))- \bint_0^t \left(\tfrac{\sigma}{2\Psi(X_0(\tau),\tau)}-\tfrac{\sigma^{2}}{4}X_0^{2}(\tau) \right)\dee\tau\]
and 
\[V(t)=\tfrac{w_m(0)}{w_m(t)}\exp\left(- \bint_0^t \left(\tfrac{\sigma}{2\Psi(X_0(\tau),\tau)}-\tfrac{\sigma^{2}}{4}X_0^{2}(\tau) \right)\dee\tau\right).\] Therefore by using \eqref{P3-4} we have
\[V(t)=\tfrac{2}{\sigma w_m(t)} v(t)\] then 
\[\tfrac{\sigma}{2}w_m(t)\exp\left( \tfrac{\sigma}{2}\bint_0^{s(t)} u(\xi,t)\dee \xi\right)=\tfrac{\sigma}{2}w_m(t)V(t)=v(t)\] and \eqref{P3-5} follows.
From Remark \ref{signosX0-X1} and \eqref{defz} we obtain \eqref{P2-5}. 
\end{proof}
Next we obtain a connection between the moving boundary problem ($P_3$) for the nonlinear transport equation incorporating a source term with a class of moving boundary problems for a nonlinear evolution equation incorporating an exponential source  term.

\begin{thm} \label{theo_3-4}
If $(\Psi(x,t),X_0(t),X_1(t))$ is a solution to the problem ($P_3$) given by  \eqref{P3-1}-\eqref{P3-5}, then $(\theta(y,t),Y_0(t),Y_1(t))$ where
\begin{equation} \label{transf theta}
\theta(y,t)=(1+mx) \Psi(x,t),\qquad \qquad y=\tfrac{\ln(1+mx)}{m},
\end{equation}
\begin{equation} \label{transf Y0Y1}
    Y_0(t)=\tfrac{\ln(1+mX_0(t))}{m},\qquad \qquad Y_1(t)=\tfrac{\ln(1+mX_1(t))}{m},
\end{equation}
 for a fixed parameter $m>0$, satisfies the  problem ($P_4$) given by 
 \begin{subequations}
\begin{align}
& \theta_t = -(\theta^{-1})_{yy} - m(\theta^{-1})_y + \sigma \exp(my), \qquad Y_0(t) < y < Y_1(t), & t > 0, \label{P4-1} \\[0.35cm]
& Y_1(t) = \frac{\ln(1 + m I_1(t))}{m}, & t > 0, \label{P4-2} \\[0.35cm]
& \theta(Y_1(t), t) = \exp(m Y_1(t)) \left(-\tfrac{2}{\sigma} \tfrac{w_m'(t)}{w_m(t)} + \tfrac{\sigma (\exp(m Y_1(t)) - 1)^2}{2 m^2} \left(1 - \tfrac{w_m(t)}{L(t)}\right)\right)^{-1}, & t > 0, \label{P4-3} \\[0.35cm]
& \int_0^t \left(\tfrac{\sigma^2(\exp(m Y_0(\tau)) - 1)^2}{4m^2} - \tfrac{\sigma \exp(m Y_0(\tau))}{2 \theta(Y_0(\tau), \tau)}\right) d\tau = \ln\left(\frac{2v(t)}{\sigma w_m(0)}\right), & t > 0, \label{P4-4} \\[0.35cm]
& Y_0(t) = \frac{\ln(1 + m I_0(t))}{m}, & \label{P4-5}
\end{align}
\end{subequations}
with
 \begin{align}
&I_0(t)=- \bint_0^t \exp(mY_0(\tau))\left( \tfrac{\sigma (\exp(m Y_0(\tau))-1) }{m\theta(Y_0(\tau),\tau)}+\tfrac{ \theta_y(Y_0(\tau),\tau)-m \theta(Y_0(\tau),\tau)}{\theta^3(Y_0(\tau),\tau)} \right)\, d \tau, \label{I0}\\[0.35cm]
&I_1(t)=\bint_0^t \exp(mY_1(\tau))\left(\tfrac{\sigma \left(\exp(m Y_1(\tau)-1)\right)}{m\theta(Y_1(\tau),\tau)}\left(\tfrac{ w_m(\tau)}{2 L(\tau)}-1\right) -\tfrac{\theta_y(Y_1(\tau),\tau)-m\theta(Y_1(\tau),\tau)}{\theta^{3}(Y_1(\tau),\tau)}\right) d \tau.\label{I1}
\end{align}
\end{thm}

\begin{proof}
Notice that transformation \eqref{transf theta} is well-defined due to the fact that $x\geq 0$ and then $1+mx>0$ for $m>0$. Through simple calculations we get that
    \begin{equation}\label{Psix}
    \Psi(x,t)=\exp(-my)\theta(y,t), \qquad x=\tfrac{\exp(my)-1}{m},
    \end{equation}
    and taking into account that $y_x=\tfrac{1}{1+mx}=\exp(-my)$, we also obtain that
\begin{equation}
    \begin{array}{ll}
    \Psi_t(x,t)=\exp(-my)\theta_t(y,t),\\[0.35cm]
        \left(\Psi^{-1}\right)_x(x,t)=m \theta^{-1}(y,t)+\left(\theta^{-1}\right)_y(y,t), \\[0.35cm]
        \left(\Psi^{-1}\right)_{xx}(x,t)=\left(m \left(\theta^{-1}\right)_y(y,t)+\left(\theta^{-1}\right)_{yy}(y,t) \right) \exp(-my).
    \end{array}
\end{equation}
 As a consequence, if the differential equation \eqref{P3-1} for \(\Psi\) is valid for \(X_0(t) < x < X_1(t)\) and \(t > 0\), then we obtain:
  $$\exp(-my)\theta_t(y,t)=\left(m \left(\theta^{-1}\right)_y(y,t)+\left(\theta^{-1}\right)_{yy}(y,t) \right) \exp(-my)+\sigma.$$
This implies that equation \eqref{P4-1} is satisfied for \(Y_0(t) < y < Y_1(t)\), where \(Y_0\) and \(Y_1\) are defined by \eqref{transf Y0Y1}.

Under the transformation \eqref{transf theta} we also have that
  \begin{equation}\label{XsobrePsi}
     \tfrac{x}{\Psi(x,t)}=\tfrac{\exp(my)(\exp(my)-1)}{m\theta(y,t)},
  \end{equation}
  and
  \begin{equation}\label{PsixsobrePsi3}
      \tfrac{\Psi_x(x,t)}{\Psi^3(x,t)}=\tfrac{\exp(my) \left( \theta_y(y,t)-m\theta(y,t)\right)}{\theta^3(y,t)}.
  \end{equation}
Subsequently, based on condition \eqref{P3-2} and considering \eqref{transf Y0Y1}, \eqref{XsobrePsi}, and \eqref{PsixsobrePsi3}, it can be inferred that
 $$ \tfrac{\exp(mY_1(t))-1}{m}=\bint_0^t \exp(mY_1(\tau))\left(\tfrac{\sigma \left(\exp(m Y_1(\tau)-1)\right)}{m\theta(Y_1(\tau),\tau)}\left(\tfrac{ w_m(\tau)}{2 L(\tau)}-1\right) -\tfrac{\theta_y(Y_1(\tau),\tau)-m\theta(Y_1(\tau),\tau)}{\theta^{3}(Y_1(\tau),\tau)}\right) d \tau  $$
an so   condition \eqref{P4-2} holds, i.e $Y_1(t)=\tfrac{\ln(1+mI_1(t))}{m}$ where $I_1$ is defined by \eqref{I1}.

Condition \eqref{P4-3} follows straightforwardly from \eqref{P3-3}, \eqref{transf Y0Y1} and \eqref{Psix}.

Moreover, by considering \eqref{P3-4}, \eqref{transf Y0Y1}, and \eqref{Psix}, it becomes clear that condition \eqref{P3-4} is satisfied.

Finally,  taking into account \eqref{transf Y0Y1}, \eqref{XsobrePsi}, and \eqref{PsixsobrePsi3}, condition \eqref{P3-5} leads to
$$ \tfrac{\exp(mY_0(t))-1}{m}=- \bint_0^t \exp(mY_0(\tau))\left( \tfrac{\sigma (\exp(m Y_0(\tau))-1) }{m\theta(Y_0(\tau),\tau)}+\tfrac{ \theta_y(Y_0(\tau),\tau)-m \theta(Y_0(\tau),\tau)}{\theta^3(Y_0(\tau),\tau) \exp(m Y_0(\tau))} \right)\, d \tau,$$
from which we deduce that $Y_0(t)=\tfrac{\ln(1+mI_0(t))}{m}$ where $I_0$ is defined by \eqref{I0}, i.e. condition \eqref{P4-5} holds.
\end{proof}

\begin{thm}
If $(\theta,Y_0,Y_1)$ is a solution  to the problem ($P_4$) given by \eqref{P4-1}-\eqref{P4-5} for $m>0$ then 
$(\Psi,X_0,X_1)$  with
\begin{equation} \label{transf psi}
\Psi(x,t)=\exp(-my)\theta(y,t),\qquad  \qquad x=\tfrac{\exp(my)-1}{m},
\end{equation}
\begin{equation} \label{transf X0X1}
   X_0(t)=\tfrac{\exp(mY_0(t))-1}{m},\qquad \qquad X_1(t)=\tfrac{\exp(mY_1(t))-1}{m},
\end{equation}
is a solution to the problem ($P_3$) given by  \eqref{P3-1}-\eqref{P3-5} .
 \end{thm}

\begin{proof}
From \eqref{transf psi} we have that
\begin{equation}
    \exp(my)=1+mx,
\end{equation}
\begin{equation}\label{thetapsi-demo}
    \theta(y,t)=(1+mx)\Psi(x,t).
\end{equation}
    Taking into account that $x_y=1+mx$ it results that
    \begin{equation}
        \begin{array}{lll}
             \theta_t(y,t)=(1+mx)\Psi_t(x,t),\\[0.35cm]
             \left(\theta^{-1}\right)_y(y,t)=-\tfrac{m}{1+mx} \Psi^{-1}(x,t)+\left( \Psi^{-1}\right)_x(x,t),\\[0.35cm]
              \left(\theta^{-1}\right)_{yy}(y,t)=\tfrac{m^2}{1+mx}\Psi^{-1}-m\left(\Psi^{-1} \right)_x+(1+mx)\left(\Psi^{-1} \right)_{xx}.
        \end{array}
    \end{equation}
    Therefore, if \eqref{P4-1} holds for $Y_0(t)<y<Y_1(t)$ and $t>0$ then we immediately get \eqref{P3-1} for $X_0(t)<x<X_1(t),\; t>0$ where $X_0$ and $X_1$ are given by \eqref{transf X0X1}. 

   Based on \eqref{P4-2} and  \eqref{transf X0X1} we deduce that $X_1(t)=I_1(t)$ where $I_1$ is given by \eqref{I1}. Considering \eqref{transf psi} it follows that
   \begin{equation}\label{aux-P4P3}
   \begin{array}{ll}
       \tfrac{\exp(my)}{\theta(y,t)}=\Psi^{-1}(x,t)\\[0.35cm]
              \tfrac{\exp(my)(\theta_y(y,t)-m\theta(y,t))}{\theta^3(y,t)}=\tfrac{\Psi_x(x,t)}{\Psi^3(x,t)},
   \end{array}
   \end{equation}
 and  subsequently we obtain that
   $$X_1(t)=\bint_0^t \left( \tfrac{\sigma X_1(\tau)}{\Psi(X_1(\tau),\tau)} \left(\tfrac{ w_m(\tau)}{2L(\tau)}-1\right)-\tfrac{\Psi_x(X_1(\tau),\tau)}{\Psi^3(X_1(\tau),\tau)}\right) \, d\tau.$$
   As a consequence, condition \eqref{P3-2} is valid.

Taking into account \eqref{transf X0X1} and \eqref{thetapsi-demo}, condition \eqref{P3-3} can be derived directly from \eqref{P4-3}.

Furthermore, considering condition \eqref{P4-3}, \eqref{transf X0X1} and \eqref{aux-P4P3}  leads to \eqref{P3-3}.

Finally,  from \eqref{P4-5} and  \eqref{transf X0X1} it follows that $X_0(t)=I_0(t)$ where $I_0$ is given by \eqref{I0}. Hence, by \eqref{aux-P4P3} we obtain that
$$X_0(t)=- \bint_0^t \left(\tfrac{\sigma X_0(\tau)}{\Psi(X_0(\tau),\tau)} +\tfrac{\Psi_x(X_0(\tau),\tau)}{\Psi^3(X_0(\tau),\tau)} \right)\; d \tau,$$
which means that \eqref{P3-5} holds.
\end{proof}
The following theorem links Stefan's problem $(P_1)$ with the problem $(P_4)$ and allows to parameterize the explicit solution.
\begin{thm} \label{sol}
   If $(w,s)$ is a solution to the problem $(P_1)$ given by \eqref{P1-1}-\eqref{P1-5} then a parametric explicit solution to the problem $(P_4)$ is given by 
    
\begin{equation}\theta(y,t)=\frac{mw_z(z,t)w(z,t)-\tfrac{\sigma}{2} w^{2}(z,t)}{w_{zz}(z,t)w(z,t)-w_z^{2}(z,t)},\quad Y_0(t)\leq y\leq Y_1(t),\quad t>0,\label{sol_1}\end{equation}

  \begin{equation}
 y=\tfrac{1}{m}\ln\left(1-\tfrac{2mw_z(z,t)}{\sigma w(z,t)}\right)  
 \quad 0\leq z\leq s(t),\quad t>0,\label{sol_2} \end{equation} where the free boundaries $Y_0$ and $Y_1$ are given by 
\begin{equation}Y_0(t)= \tfrac{1}{m}\ln\left(1-\tfrac{2mw_z(0,t)}{\sigma w(0,t)}\right)\quad \textit{and}\quad Y_1(t)= \tfrac{1}{m}\ln\left(1-\tfrac{2mw_z(s(t),t)}{\sigma w(s(t),t)}\right),\quad t>0. \label{sol_3}
  \end{equation}
\end{thm}
\begin{proof}
    Assuming there exists $(w(z,t),s(t))$ a solution of the Stefan type problem $(P_1)$ given by equations \eqref{P1-1}-\eqref{P1-5}, through \eqref{transf x-w} we obtain a solution $(x,s)$ to the problem $(P_2)$ defined by \eqref{P2-1}-\eqref{P2-5}. Then, by using the reciprocal transformation given by \eqref{T3} we have
    \begin{equation}\Psi(x,t)=-\tfrac{\sigma}{2} \tfrac{w^{2}(z,t)}{w_{zz}(z,t)w(z,t)-w_z^{2}(z,t)},\quad X_0(t)\leq x\leq X_1(t),\quad t>0,
    \end{equation}
    \begin{equation}
 x=-\tfrac{2}{\sigma}\tfrac{w_z(z,t)}{w(z,t)}
 \quad 0\leq z\leq s(t),\quad t>0, \end{equation} and the free boundaries $X_0$ and $X_1$ are given by
 \[X_0(t)=-\tfrac{2}{\sigma}\tfrac{w_z(0,t)}{w(0,t)}, \qquad\qquad X_1(t)=-\tfrac{2}{\sigma}\tfrac{w_z(s(t),t)}{w(s(t),t)}, \qquad t>0.\]
 Finally, by \eqref{transf theta} and \eqref{transf Y0Y1} we obtain the expressions for $(\theta(y,t),Y_0(t),Y_1(t))$ as function of $(w(z,t),s(t)).$
    
\end{proof}

\subsection{Stefan type problem with Robin-Neumann conditions}
In this subsection we consider a Stefan type problem ($P_1^*$) analogous to the one given by \eqref{P1-1}-\eqref{P1-5}, in which we replace the fixed boundary condition given by \eqref{P1-4} by the following condition 
\begin{equation}\label{convectiva}
w_z(0,t)=h(t)(\varepsilon w(0,t)-v(t)),\qquad t>0
\tag{\ref{P1-4}$^*$}
\end{equation}
for $\varepsilon=0$ or $\varepsilon=1$, where $h=h(t)$ is a positive and continuous function which characterises the heat transfer at the fixed face $z=0$ for $t>0$.\\
Notice that for $\varepsilon = 0$ the condition \eqref{convectiva} corresponds to a Neumann condition, and for $\varepsilon = 1$ refers to a Robin condition.\\
Under the transformations \eqref{transf x-w}, \eqref{T3}-\eqref{XceroXuno} and \eqref{transf theta}-\eqref{transf Y0Y1} from \eqref{convectiva} we get the conditions
\begin{equation}\label{convectiva_P2}
h(t)v(t) = [h(t)\varepsilon + \tfrac{\sigma}{2}x(0,t)]w_m(t)\exp\left(\tfrac{\sigma}{2}\bint_0^{s(t)}x(\xi,t)\,d\xi\right),\qquad t>0,
\tag{\ref{P2-4}$^*$}    
\end{equation}

\begin{equation}\label{convectiva_P3}
h(t)v(t) = [h(t)\varepsilon + \tfrac{\sigma}{2}X_0(t)]w_m(0)\exp\left(\bint_0^t \left(\tfrac{\sigma^2}{4}X_0^2(\tau)-\tfrac{\sigma}{2}\tfrac{1}{\Psi(X_0(\tau),\tau)} \right) \dee \tau \right),\qquad t>0,
\tag{\ref{P3-4}$^*$}    
\end{equation}

\begin{equation}\label{convectiva_P4}
h(t)v(t) = [h(t)\varepsilon + \tfrac{\sigma(\exp(mY_0(t))-1)}{2m}]w_m(0)\exp\left(\bint_0^t \left(\tfrac{\sigma^2(\exp(mY_0(t))-1)^2}{4m^2}-\tfrac{\sigma\exp(mY_0(\tau))}{2\theta(Y_0(\tau),\tau)} \right) \dee \tau \right),\qquad t>0.
\tag{\ref{P4-4}$^*$}    
\end{equation}\\
which replace \eqref{P2-4}, \eqref{P3-4} and \eqref{P4-4} in problems ($P_2$), ($P_3$) and ($P_4$) respectively. Let us call these new problems as ($P_2^*$), ($P_3^*$) and ($P_4^*$).\\

For the free boundary problems ($P_1^*$) and ($P_4^*$) we see that Theorem \ref{sol} is also fulfilled and can be established as follows:\\
\begin{thm} \label{sol.1}
   If $(w^{\varepsilon},s^{\varepsilon})$ is a solution to the problem $(P_1^*)$ for $\varepsilon = 0$ or $\varepsilon = 1$, then  a parametric explicit solution to the problem $(P_4^*)$ is given by 
    
\begin{equation}\theta^{\varepsilon}(y,t)=\frac{mw_z^{\varepsilon}(z,t)w^{\varepsilon}(z,t)-\tfrac{\sigma}{2}[w^{\varepsilon}(z,t)]^2}{w_{zz}^{\varepsilon}(z,t)w^{\varepsilon}(z,t)-[w_z^{\varepsilon}(z,t)]^2},\quad Y_0^{\varepsilon}(t)\leq y\leq Y_1^{\varepsilon}(t),\quad t>0,\label{sol_1.1}\end{equation}

  \begin{equation}
 y=\tfrac{1}{m}\ln\left(1-\tfrac{2mw_z^{\varepsilon}(z,t)}{\sigma w^{\varepsilon}(z,t)}\right)  
 \quad 0\leq z\leq s^{\varepsilon}(t),\quad t>0,\label{sol_2.1} \end{equation} where the free boundaries are given by 
\begin{equation}Y_0^{\varepsilon}(t)= \tfrac{1}{m}\ln\left(1-\tfrac{2mw_z^{\varepsilon}(0,t)}{\sigma w^{\varepsilon}(0,t)}\right)\quad \textit{and}\quad Y_1^{\varepsilon}(t)= \tfrac{1}{m}\ln\left(1-\tfrac{2mw_z^{\varepsilon}(s^{\varepsilon}(t),t)}{\sigma w^{\varepsilon}(s^{\varepsilon}(t),t)}\right),\quad t>0.  
  \end{equation}
\end{thm}
\begin{proof}
   The proof is analogous to the one provided in Theorem \ref{sol}.
\end{proof}

In the situation where there exists a unique solution to the problems \((P_i)\) and \((P_i^*)\), we will present a convergence result demonstrating that the solution of problem \((P_i^*)\) converges to the solution of problem \((P_i)\) as \(h\) approaches infinity, specifically for \(\varepsilon = 1\).
\begin{thm} \label{epsilon-0}
 Let \(\{h_n\}_{n \in \mathbb{N}}\) denote a sequence of continuous and positive functions such that for each \(t > 0\), \(\lim_{n \to +\infty} h_n(t) = +\infty\). We define \((P_{in}^*)\) as the problem associated with \((P_i^*)\) when \(h = h_n\) in \eqref{convectiva} for \(\varepsilon = 1\).

If there exist unique solution to the problem ($P_{in}^*$), then that solution converges to the unique solution of ($P_i$) when $n\rightarrow +\infty$ for $i=1, 2, 3, 4$. 
\end{thm}
\begin{proof}
    If $(w_n^1,s_n^{1})$ is the solution of the problem ($P_{1n}^*$), from condition \eqref{convectiva} we have that $h_n\to +\infty$ implies that $w_n^1(0,t) \rightarrow v(t)$ when $n\to +\infty$, for each $t>0$.
    Therefore the pair $(w,s)$ where $ w(x,t)=  
    \lim_{n\to +\infty} w_n^1(x,t)$ and $s(t)=\lim_{n\to +\infty} s_n^1(t) $ 
    satisfy \eqref{P1-1}-\eqref{P1-5}, this means that it is the solution of ($P_1$). 
    Therefore ($P_{1n}^*$)$\to$ ($P_1$) for $n\to +\infty$.\\
    This convergence implies that for $n\to +\infty$ results that $h_n\to +\infty$ and then the solution of the problems $P_{2n}^*$, $P_{3n}^*$ and $P_{4n}^*$ converges to the solution of the problems $P_2$, $P_3$ and $P_4$ respectively.
\end{proof}

\section{Explicit solutions for particular cases}

We will analyze the existence of the solutions to the problems $P_1$ and $P^*_1$ for the particular case when the functions $L(t)$, $v(t)$, $h(t)$ and $w_m(t)$ are assumed as follows
\be
L(t)=L_0\sqrt{t},\qquad v(t)=2v_0\sqrt{t},\qquad h(t)=\tfrac{h_0}{\sqrt{t}},\qquad w_m(t)=2w_{m_0}\sqrt{t},\label{p_cond}
\ee
where $L_0$, $v_0$, $h_0$ and $w_{m_0}$ are non-negative constants with $v_0>w_{m_0}$.

\subsection{Particular case with a Dirichlet condition}

In this subsection, we examine the problem (P1) that involves a Dirichlet boundary condition at the fixed face $x=0$. The problem described by equations \eqref{P1-1}-\eqref{P1-5} is reformulated as follows:\begin{subequations}\label{dirich_part}
\begin{align}
&w_{zz}=w_t,\qquad &0<z<s(t),&\qquad t>0,\label{P´1-1}\\
&w_z(s(t),t)=-L_0\sqrt{t}s'(t),& &\qquad t>0,\label{P´1-2}\\
&w(s(t),t)=2w_{m_0}\sqrt{t},& &\qquad t>0,\label{P´1-3}\\
&w(0,t)=2v_0\sqrt{t},& &\qquad t>0,\label{P´1-4}\\
&s(0)=0\label{P´1-5}
\end{align}
\end{subequations}
and we can state the following theorem of existence and uniqueness of solution.
\begin{thm}
There exists a unique similarity type solution $\big( w, s \big)$ to the problem \eqref{P´1-1}-\eqref{P´1-5} which is given by:
    \begin{align}
        &w(z,t)= 2v_0\sqrt{t}\exp{\left(-\tfrac{z^2}{4t}\right)} + v_0z\sqrt{\pi}\erf{\left(\tfrac{z}{2\sqrt{t}}\right)} - L_0z\gamma^* - v_0z\sqrt{\pi}\erf{(\gamma^*)},\quad 0\leq z\leq s(t),\quad t>0, \label{w_solution}\\
        &s(t)=2\gamma^*\sqrt{t},\quad t>0  \label{frontera_p_solution}
    \end{align}
    where $\gamma^*$ is the unique solution to the equation given by
    $$v_0\exp(-\gamma^2) = L_0\gamma^2 + w_{m_0},\qquad \gamma> 0.$$
\end{thm}

\begin{proof}
    To solve \eqref{P´1-1}-\eqref{P´1-5} 
    we propose the following similarity transformation
    \begin{equation}
        T(\eta)=\tfrac{1}{2\sqrt{t}} w(z,t),\qquad \eta=\tfrac{z}{2\sqrt{t}},\label{T-t}.
    \end{equation}
    From this transformation, we find that the free boundary is given by
    \be
        s(t)=2\gamma \sqrt{t}  \label{frontera_p}
    \ee
    where $\gamma > 0$ must be determined.\\
    From \eqref{T-t}, we have that\\
    \be 
        w_z(z,t) = T'(\eta), \qquad
        w_{zz}(z,t) = \tfrac{1}{2\sqrt{t}}T''(\eta), \qquad
        w_t(z,t) = \tfrac{1}{\sqrt{t}}T(\eta) - \tfrac{\eta}{\sqrt{t}}T'(\eta). \label{wt-trans}
    \ee \\
    Using \eqref{wt-trans} in \eqref{P´1-1} we get
    $$T''(\eta)+2\eta T'(\eta)-2T(\eta)=0,\qquad 0<\eta<\gamma.$$
   Taking into account \eqref{P´1-2}, \eqref{frontera_p} and \eqref{wt-trans} we have\\
    $$T'(\gamma)=-L_0\gamma.$$
    From \eqref{P´1-3}, \eqref{T-t} and \eqref{frontera_p} it follows
    $$T(\gamma)=w_{m_0},$$
    besides, from \eqref{P´1-4} and \eqref{T-t} we obtain
    $$T(0)=v_0.$$
    Then, the problem \eqref{P´1-1}-\eqref{P´1-5} 
    is equivalent to the following problem given by
    \begin{subequations} \label{T-dir-prob}
    \begin{align} 
        &T''(\eta)+2\eta T'(\eta)-2T(\eta)=0,\qquad &0<\eta<\gamma,\label{T-ecdif}\\
        &T'(\gamma)=-L_0\gamma,\label{T-ecdif-2}\\
        &T(\gamma)=w_{m_0}, \label{T-ecdif-3}\\
        &T(0)=v_0. \label{T-cond en 0}
    \end{align}
    \end{subequations}
    The solution of \eqref{T-ecdif} was obtained in \cite{SaTa2011-a} and it is given by
    \begin{align}
      T(\eta)=A\Big[ \exp(-\eta^2)+\sqrt{\pi} \eta \erf(\eta)\Big]+B\eta, \label{T_sol}  
    \end{align}
    where $A$ and $B$ are determined from conditions \eqref{T-cond en 0} and \eqref{T-ecdif-2}. We have
    $$A = v_0 \qquad , \qquad B = -L_0\gamma - v_0\sqrt{\pi}\erf(\gamma).$$
    Using \eqref{T-ecdif-3}, we find that the coefficient $\gamma$ must satisfy the equation:
    \begin{align}
        &v_0\exp(-\gamma^2) = L_0\gamma^2 + w_{m_0} \label{gamma}.
    \end{align}\\
    Defining $G(\gamma) = v_0\exp(-\gamma^2)$ and $F(\gamma) = L_0\gamma^2 + w_{m_0}$ we have
    $$G(0) = v_0, \qquad G(+\infty) = 0, \qquad G'(\gamma) < 0, \qquad \gamma > 0,$$
    $$F(0) = w_{m_0}, \qquad F(+\infty) = +\infty, \qquad F'(\gamma) > 0, \qquad \gamma > 0.$$\\
    Based on the properties of the functions \( G \) and \( F \), and considering that \( v_0 > w_{m_0} \), there exists a unique value \( \gamma^* \) such that \eqref{gamma} is satisfied.

    Then, the solution to the problem \eqref{T-ecdif}-\eqref{T-cond en 0} is given by\\
    $$T(\eta)=v_0\Big[ \exp(-\eta^2)+\sqrt{\pi} \eta \erf(\eta)\Big]+\Big[ -L_0\gamma^* - v_0\sqrt{\pi}\erf(\gamma^*) \Big]\eta,$$
    where $\gamma^*$ is the unique solution to equation \eqref{gamma}.\\
    Therefore, there exists a unique similarity solution to the problem \eqref{P´1-1}-\eqref{P´1-5}
    which is expressed by \eqref{w_solution} and \eqref{frontera_p_solution}.\\
\end{proof}

\begin{obs}
    Assuming \eqref{p_cond} the problem ($P_4$) admits the parametric solution given by \eqref{sol_1}-\eqref{sol_3} where $w(x,t)$ and $s(t)$ are defined by \eqref{w_solution} and \eqref{frontera_p_solution}.
\end{obs}

\subsection{Particular case with a Robin-Neumann condition}
In this subsection we consider the Robin (Neumann) condition \eqref{convectiva} for $\varepsilon=1$ ($\varepsilon=0$) instead of \eqref{P1-4} for the functions $h(t)= \tfrac{h_0}{\sqrt{t}}$ and $v(t)= 2v_0\sqrt{t}$. We can state the following theorem regarding the  existence and uniqueness of the solution for the problem\\
\begin{subequations}\label{dirich_part*}
\begin{align}
&w_{zz}=w_t,\qquad &0<z<s(t),&\qquad t>0,\label{P*1-1}\\
&w_z(s(t),t)=-L_0\sqrt{t}s'(t),& &\qquad t>0,\label{P*1-2}\\
&w(s(t),t)=2w_{m_0}\sqrt{t},& &\qquad t>0,\label{P*1-3}\\
&w_z(0,t)=\frac{h_0}{\sqrt{t}}(\varepsilon w(0,t)-2v_0\sqrt{t}),& &\qquad t>0,\label{P*1-4}\\
&s(0)=0.\label{P*1-5}
\end{align}
\end{subequations}
\begin{thm}
There exists a unique similarity solution $\big( w^{\varepsilon}, s^{\varepsilon} \big)$ to the problem \eqref{dirich_part*} given by:
    \begin{align}
        &w^{\varepsilon}(z,t)= \tfrac{2\sqrt{t}(2h_0v_0-L_0\gamma^{\varepsilon})}{\sqrt{\pi}\erf(\gamma^{\varepsilon})+2h_0\varepsilon}\left[\exp(-\tfrac{z^2}{4t})+\tfrac{z\sqrt{\pi}}{2\sqrt{t}}\erf(\tfrac{z}{2\sqrt{t}})\right]-z\Big[\tfrac{2L_0\gamma^\varepsilon h_0\varepsilon + 2h_0v_0\sqrt{\pi}\erf(\gamma^\varepsilon)}{\sqrt{\pi}\erf(\gamma^\varepsilon)+2h_0\varepsilon}\Big],\quad t>0, \label{w_conv_solution}\\
        &s^{\varepsilon}(t)=2\gamma^{\varepsilon}\sqrt{t},\quad t>0,  \label{frontera_p_conv_solution}
    \end{align}
    where $\gamma^{\varepsilon}$ is the unique solution to
    $$(2v_0h_0-L_0\gamma)\exp(-\gamma^2) = (w_{m_0}+L_0\gamma^2)[\sqrt{\pi}\erf(\gamma)+2h_0\varepsilon], \qquad \gamma>0.$$
\end{thm}

\begin{proof}
    The proof is similar to the one in the previous section. Under the similarity transformation \eqref{T-t}-\eqref{frontera_p}, we obtain the condition $T'(0) = 2h_0[\varepsilon T(0)-v_0]$ instead of \eqref{T-cond en 0}. Therefore we have the following ordinary differential problem
    \begin{subequations}
    \begin{align}
        &T''(\eta)+2\eta T'(\eta)-2T(\eta)=0,\qquad &0<\eta<\gamma,\label{T-ecdif*}\\
        &T'(\gamma)=-L_0\gamma,\label{T-ecdif-2*}\\
        &T(\gamma)=w_{m_0}, \label{T-ecdif-3*}\\
        &T'(0) = 2h_0[\varepsilon T(0)-v_0],\qquad t>0, \label{T_convectiva_part}
    \end{align}
    \end{subequations}
    whose solution is 
    \begin{align}
      T(\eta)=C\Big[ \exp(-\eta^2)+\sqrt{\pi} \eta \erf(\eta)\Big]+D\eta \label{T_sol_CD}  
    \end{align} where the constants $C$ and $D$ need to be determined.
    
    By \eqref{T_sol_CD} we have that $T(0) = C$ and $T'(0)= D$, that take us with \eqref{T_convectiva_part} to $D = 2h_0[\varepsilon C-v_0]$. Using \eqref{T-ecdif-2*} we obtain 
    $$C = \tfrac{2h_0v_0 - L_0\gamma}{\sqrt{\pi}\erf(\gamma)+2h_0\varepsilon} \qquad\textit{and}\qquad D =-\tfrac{2L_0\gamma h_0\varepsilon + 2h_0v_0\sqrt{\pi}\erf(\gamma)}{\sqrt{\pi}\erf(\gamma)+2h_0\varepsilon}.$$
    According to \eqref{T-ecdif-3*} we obtain for $\gamma$ the following equation
    $$\tfrac{2h_0v_0-L_0\gamma}{\sqrt{\pi}\erf(\gamma)+2h_0\varepsilon}[\exp(-\gamma^2)+\gamma\sqrt{\pi}\erf(\gamma)]-\gamma\left[\tfrac{2L_0\gamma h_0\varepsilon + 2h_0v_0\sqrt{\pi}\erf(\gamma)}{\sqrt{\pi}\erf(\gamma)+2h_0\varepsilon}\right] = w_{m_0}$$
    which is equivalent to
    \begin{align}
        & (2v_0h_0-L_0\gamma)\exp(-\gamma^2) = (w_{m_0}+L_0\gamma^2)[\sqrt{\pi}\erf(\gamma)+2h_0\varepsilon].\label{gamma_convec}
    \end{align}\\
    Assuming $G(\gamma) = (2v_0h_0-L_0\gamma)\exp(-\gamma^2)$ and $F(\gamma) = (w_{m_0}+L_0\gamma^2)[\sqrt{\pi}\erf(\gamma)+2h_0\varepsilon]$ we have
    $$G(0) = 2v_0h_0, \qquad G(\tfrac{2v_0h_0}{L_0}) = 0, \qquad G'(\gamma) < 0, \qquad  0<\gamma\leq \tfrac{2v_0h_0}{L_0},$$
    $$F(0) = 2w_{m_0}h_0\varepsilon, \qquad F(+\infty) = +\infty, \qquad F'(\gamma) > 0, \qquad \gamma > 0.$$\\
    From properties of functions $G$ and $F$ we obtain that if $v_0>\varepsilon w_{m_0}$, then there exists a unique $\gamma^\varepsilon$ such that \eqref{gamma_convec} holds.
    Then the solution to the problem \eqref{T-ecdif*}-\eqref{T_convectiva_part} is given by\\
    $$T(\eta)=\tfrac{2h_0v_0-L_0\gamma^\varepsilon}{\sqrt{\pi}\erf(\gamma^\varepsilon)+2h_0\varepsilon}\Big[ \exp(-\eta^2)+\eta\sqrt{\pi}\erf(\eta)\Big] -\eta\Big[\tfrac{2L_0\gamma^\varepsilon h_0\varepsilon + 2h_0v_0\sqrt{\pi}\erf(\gamma^\varepsilon)}{\sqrt{\pi}\erf(\gamma^\varepsilon)+2h_0\varepsilon}\Big],$$
    where $\gamma^\varepsilon$ is the unique solution to equation \eqref{gamma_convec}.\\
    Therefore, there exists a unique similarity solution to the problem \eqref{P*1-1}-\eqref{P*1-5} which is given by \eqref{w_conv_solution} and \eqref{frontera_p_conv_solution}.\\
\end{proof}

\begin{obs}
    
Assuming \eqref{p_cond} the problem ($P_4^*$) has a parametric explicit solution given by \eqref{sol_1}-\eqref{sol_3} where $w^\varepsilon(x,t)$ and $s^\varepsilon(t)$ are defined by \eqref{w_conv_solution} and \eqref{frontera_p_conv_solution}.

\end{obs}

\begin{obs}
According to Theorem \ref{epsilon-0}, for the particular case given in this subsection it is easy to check that, for $\varepsilon = 1$ we have that the solution of \eqref{P*1-1}-\eqref{P*1-5} given by \eqref{w_conv_solution} and \eqref{frontera_p_conv_solution} converges to the solution of \eqref{P´1-1}-\eqref{P´1-5} given  by \eqref{w_solution}-\eqref{frontera_p_solution} when the coefficient $h_0 \to +\infty$.
\end{obs}

\section{Conclusions}
In this work, we investigated a class of moving boundary problems associated with a nonlinear evolution equation featuring an exponential source term. Through the application of reciprocal transformations and the Cole-Hopf transformation, we established a significant connection to Stefan-type problems, thereby enhancing our understanding of the underlying dynamics involved in these systems.

Our results reveal that the boundary conditions at the fixed face play a crucial role in solution formation. By deriving explicit similarity solutions in parametric form for specific cases, we have provided a theoretical framework that not only describes the behavior of the solutions but also facilitates the identification of patterns and trends in complex dynamic systems.

This innovative approach not only contributes to the theory of moving boundary problems but also offers practical insights for applications across various fields of science and engineering. Future research could delve deeper into generalizing our solutions and exploring additional effects, such as the influence of external perturbations and non-homogeneous initial conditions.

\begin{appendices}
\section*{Appendix: A canonical solitonic reciprocal reduction}

 Here, a 3rd order nonlinear evolution equation with exponential source term is introduced, namely
 \begin{equation}\tag{A1}\label{a1}
     \Psi_t=\frac{1}{2}\left[ \left( \Psi^{-2}\right)_{xxx}-\left(\Psi^{-2} \right)_x \right] +\alpha \exp(mx)+\beta \exp(-mx).
 \end{equation}
 Under the reciprocal transformation
 \begin{equation}\tag{A2}\label{a2}
     dy=\Psi \; dx+ \left( \frac{1}{2}\left[ \left( \Psi^{-2}\right)_{xx}-\Psi^{-2} \right]+ \frac{\alpha}{m}\exp(mx)-\frac{\beta}{m}\exp(-mx) \right) \; dt, \qquad (m\neq 0)
     \end{equation}
 it is seen that, with $\rho=\Psi^{-1}$
 \begin{equation}\tag{A3}
     dx=\rho \; dy- \rho \left( \frac{1}{2}\left[ \left( \Psi^{-2}\right)_{xx}-\Psi^{-2} \right]+ \frac{\alpha}{m}\exp(mx)-\frac{\beta}{m}\exp(-mx) \right) \; dt.
 \end{equation}
 Thus,
 \begin{equation}\tag{A4}
     x_y=\rho, \qquad x_t=-\rho_{yy}+\frac{1}{2} \rho^3-\rho \left( \frac{\alpha}{m}\exp(mx)-\frac{\beta}{m}\exp(-mx) \right), 
 \end{equation}
 so that, with $t\to -t$
 \begin{equation}\tag{A5}
     x_t=x_{yyy}-\frac{1}{2}x_y^3+x_y \left( \frac{\alpha}{m}\exp(mx)-\frac{\beta}{m}\exp(-mx)\right).
 \end{equation}
 The latter with $m=2$, $\alpha=-\beta=3$ constitutes the canonical $m^2$KdV (modofied modified Korteweg-de Vries equation) equation \cite{Fo1980,CaDe1985}. This, in turn, is reciprocally related to the solitonic torsion evolution equation as derived in \cite{ScRo1999} in the geometric context of binormal motion of an inextensible curve.

 It is remarked that, in the specialization $\alpha=\beta=0$ in \eqref{a1}, reduction via \eqref{a2} is made to the canonical mKdV equation in $v=x_y$. In recent work  \cite{Ro2023}, moving boundary problems of Stefan-type for both the mKdV equation and a reciprocally related base  Casimir member of a compacton hierarchy, namely \eqref{a1} with $\alpha=\beta=0$, \cite{OlRo1996}, have been shown to admit exact solution via Painlevé II reduction.
\end{appendices}

\section*{Acknowledgement}
The study was sponsored by the projects  O06-24CI1901 and  O06-24CI1903  from Austral University, Rosario.

\end{document}